\documentclass[10pt,a4paper]{amsart}

\usepackage{amssymb}
\usepackage{graphicx}
\usepackage{epstopdf}
\usepackage[ps2pdf]{hyperref}  
\hypersetup{colorlinks=true, linkcolor=blue, anchorcolor=blue,
citecolor=red, filecolor=blue, menucolor=blue, pagecolor=blue,
urlcolor=blue}
\numberwithin{equation}{section}
\newtheorem{theorem}{Theorem}[section]

\newtheorem{lemma}[theorem]{Lemma}
\theoremstyle{remark}
\newtheorem{remark}{Remark}[section]

\theoremstyle{definition}
\newtheorem{definition}[theorem]{Definition}

\newcommand{\xx}{\langle x\rangle}

\newcommand{\HH}{{\mathcal{H}}}

\begin{document}

\title%
[unbounded potentials]%
{Smoothing estimates for the Schr\"odinger equation with unbounded
potentials}
\begin{abstract}
  We prove a local in time smoothing estimate for a magnetic
  Schr\"o\-dinger
  equation with coefficients growing polynomially
  at spatial infinity. The assumptions on the magnetic field
  are gauge invariant and involve only the first two derivatives.
  The proof is based on the multiplier method and no pseudofferential
  techniques are required.
\end{abstract}
\date{\today}    
\author{Piero D'Ancona}
\address{Piero D'Ancona and Luca Fanelli:
SAPIENZA - Universit\`a di Roma,
Dipartimento di Matematica,
Piazzale A.~Moro 2, I-00185 Roma, Italy}
\email{dancona@mat.uniroma1.it, fanelli@mat.uniroma1.it}

\author{Luca Fanelli}

\maketitle


\section{Introduction}\label{sec:intro} 

Smoothing properties of dispersive equations have become a standard
tool in the study of nonlinear problems.
For the Schr\"odinger flow
on $\mathbb{R}^{n}$ the basic smoothing estimate is the following:
\begin{equation}\label{eq:smoo0}
  \|\xx^{-s}|D|^{1/2}e^{it \Delta}f\|_{L^{2}L^{2}}\lesssim
  \|f\|_{L^{2}},\qquad s>1/2.
\end{equation}
Here as usual the symbol $A\lesssim B$ means
$A\le CB$ for some absolute constant $C$, $\xx=(1+|x|^{2})^{1/2}$
and $|D|^{r}f=\mathcal{F}^{-1}(|\xi|^{r}\widehat{f}(\xi))$.
With $L^{2}L^{2}$ we denote the space
$L^{2}(\mathbb{R}_{t};L^{2}(\mathbb{R}^{n}_{x}))$.

In the form \eqref{eq:smoo0} the estimate was proved by
Ben-Artzi and Klainerman
\cite{Ben-ArtziKlainerman92-a}
and Chihara
\cite{Chihara02-a},
but it can be traced back at least
as far as the work of Kato on $H$-smoothing
\cite{Kato65-a}
and subsequent works of Kato-Yajima, Vega, Sj\"olin,
Constantin-Saut
\cite{KatoYajima89-a},
\cite{Vega88-a},
\cite{Sjolin87-a},
\cite{ConstantinSaut89-a}.
In view of its importance, especially for the applications
to the derivative NLS, it has been extended and improved in a
variety of directions
(see e.g~
\cite{Watanabe91-a},
\cite{KenigPonceVega91-a},
\cite{Walther99-a},
\cite{RuzhanskySugimoto06-a}).
We recall also the close
connection of this property with the Morawetz estimates
for the wave and Klein-Gordon equation, which play a central
role in scattering theory.
The gain of $1/2$ derivative, at least
on a \emph{bounded} time interval $[-T,T]$, is a quite
general phenomenon, extending to Schr\"odinger equations on manifolds
and with variable coefficients. In these general situations, it is
well known that smoothing holds as long as the metric has no trapped
rays.

A more precise way to express smoothing is using
a Morrey-Campanato type norm:
\begin{equation}\label{eq:smoo1}
  \sup_{R>0}\frac1R\int_{-\infty}^{+\infty}dt
  \int_{|x|\le R}|\nabla e^{it \Delta}f|^{2}\,dx\le
  C\|f\|_{\dot H^{1/2}}
\end{equation}
(see
\cite{ConstantinSaut89-a},
\cite{Sjolin87-a},
\cite{PerthameVega99-a}).
This stronger form of \eqref{eq:smoo0} can be proved by a variant of
Morawetz' multiplier method; more general pseudodifferential
techniques allow only to prove smoothing in the form \eqref{eq:smoo0}.

In the following we shall focus on the variable coefficient problem
on $\mathbb{R}_{t}\times \mathbb{R}^{n}_{x}$
\begin{equation}\label{eq:schro}
  \begin{split}
    &iu_t(t,x)-\left(\nabla-iA(t,x)\right)^2u+V(t,x)u(t,x)=0
    \\
    &u(0,x)=f(x),
  \end{split}
\end{equation}
for suitable potentials $A(t,x)\in\mathbb{R}^n$ and
$V(t,x)\in\mathbb{R}$, $n\geq3$. For this equation, in general,
one can only expect local (in time) smoothing,
where the $L^{2}L^{2}$ space is replaced by
\begin{equation*}
  L^{2}_{T}L^{2}=L^{2}([-T,T];L^{2}(\mathbb{R}^{n})),\qquad
  T>0.
\end{equation*}
This was proved by Yajima \cite{Yajima91-a} for
smooth potentials $V(t,x)$ with subquadratic growth and magnetic
potentials $A(t,x)$ with sublinear growth. This result was further
extended by Doi \cite{Doi94-a} to equations of the form
\begin{equation}\label{eq:doi}
  iu_t(t,x)-\sum(D_{j}-iA_{j}(t,x))g^{jk}(x)(D_{k}-iA_{k}(t,x))u
  +V(t,x)u(t,x)=0
\end{equation}
under suitable assumptions on the metric $g^{jk}(x)$, namely a
nontrapping condition, sufficient flatness at spatial infinity,
and uniform ellipticity
(see \cite{Doi94-a}, \cite{Doi05-a}).

It has been known for some time that the quadratic growth represents
a critical threshold for potentials. Indeed,
the fundamental solution of the Schr\"odinger propagator
corresponding to $-\Delta+V(x)$ with $V(x) \gtrsim\xx^{2+\delta} $
is nowhere $C^{1}$ and can be unbounded at infinity
\cite{MartinezYajima01-a}. This reflects in a \emph{weaker} smoothing
property of the solution; Yajima and Zhang (\cite{YajimaZhang02-b},
\cite{YajimaZhang04-a}; see also \cite{Yajima04-c})
obtained for the operator
$H=-\Delta+V(x)$, with a smooth potential
$V(x)\simeq\xx^{m}$, $m\ge2$, the estimate
\begin{equation}\label{eq:smoom}
  \int_{-T}^{T}\int_{|x|\le R}
  \left|\left\langle D\right\rangle^{\frac1m}e^{itH}f\right|^{2}\,dx\,dt
  \le C_{T,R}\|f\|_{L^{2}}.
\end{equation}
The result is sharp, in the sense that
the analogous estimate with $1/m$ replaced by $s>1/m$ is false.
More recently, Robbiano and Zuily \cite{RobbianoZuily06-a}
extended \eqref{eq:smoom} to general equations of the form \eqref{eq:doi},
with $C^{\infty}$ potentials in suitable symbol classes;
as in Doi's result, the
metric must be non trapping and sufficiently
flat at infinity, moreover the electric potential $V$ can grow at most like
$\xx^{m}$ and the magnetic potential $A$ can grow at most like
$\xx^{m/2}$, with corresponding conditions on all derivatives.

All the results mentioned so far are based on pseudodifferential
techniques. These allow to handle operators of a very general
form, but with some drawbacks:
\begin{itemize}
  \item The coefficient are required to be $C^{\infty}$,
  with conditions involving all the derivatives;
  this could probably be improved to assumptions involving a
  large enough number of derivatives.
  \item A more relevant problem is that these methods
  hide some important physical aspects; indeed, the
  assumptions on the magnetic terms are expressed in terms of
  the vector potential $A(t,x)$, while for example, in dimension
  $n=3$, the physically relevant
  quantity is the vector field $B=\text{curl}\,A$. In particular,
  the assumptions are not gauge invariant.
  \item A precise estimate like \eqref{eq:smoo1}
  for the Morrey-Campanato norm of the solution seems difficult
  to obtain uniquely by pseudodifferential methods.
\end{itemize}

Our goal here is to follow a different path and
adapt the method of multipliers to handle unbounded
potentials. Indeed, by elementary methods, we can prove a
Morrey-Campanato equivalent of \eqref{eq:smoom}, and address
at the same time some of the problems listed above.
In the present work
we shall only focus on equations of the form \eqref{eq:schro};
note that in order to study a general metric
by the multiplier method,
it is necessary to exhibit a `physical space' replacement for
the non trapping condition. This is an interesting problem in
itself and will be the subject of future work.

We shall express our assumptions on the magnetic field in
terms of $\text{curl}A$, which has the following standard extension
to general space dimension:

\begin{definition}\label{def.B}
  For any $n\geq2$ the matrix-valued
  field $B:\mathbb{R}^n\to\mathcal
  M_{n\times n}(\mathbb{R})$ is defined by
  \begin{equation*}
    B:=DA-DA^t,
    \qquad
    B_{ij}=\frac{\partial A^i}{\partial x^j}-\frac{\partial
    A^j}{\partial x^i}.
  \end{equation*}
  We also define the vector field
  $B_\tau:\mathbb{R}^n\to\mathbb{R}^n$ as follows:
  \begin{equation*}
    B_\tau=\frac{x}{|x|}B.
  \end{equation*}
\end{definition}

Of course we can rephrase the definition as
$B=dA$ with $A=\sum_{j}A^{j}dx^{j}$;
in dimension $n=3$, this reduces to
$B=\text{curl}\,A$, more precisely
  \begin{equation*}
    Bv=\text{curl}\,A\wedge v,
    \qquad
    \forall v\in\mathbb{R}^3.
  \end{equation*}
In particular, we have
  \begin{equation}\label{eq.curl}
    B_\tau=\frac{x}{|x|}\wedge\text{curl}\,A,
    \qquad n=3.
  \end{equation}
Hence $B_\tau(x)$ is the projection of $B=\text{curl}\,A$ on the
tangential space in $x$ to the sphere of radius $|x|$, for $n=3$.
Observe also that $B_\tau\cdot x=0$ for any $n\geq2$, hence $B_\tau$
is a tangential vector field in any dimension. Notice that our assumptions
on the magnetic field involve $B_{\tau}$ exclusively
(see \eqref{eq:ass2}) and hence are gauge invariant.

Our main tool will be the following (with
the notation $\nabla_{A}=\nabla-iA(t,x)$):

\textsc{Magnetic Virial Identity.} \textit{Let $u(t,x)$ be a
solution of \eqref{eq:schro}, $\phi=\phi(|x|)$ a smooth, radial, real
valued function and let $\Theta(t)=\int\phi|u|^2\,dx$.
Denoting with $V_r$ the radial derivative of $V$, $D^2\phi$ the
Hessian matrix and with $\Delta^2\phi=\Delta \Delta\phi$ the
bilaplacian of $\phi$, we have
\begin{equation}\label{eq:identity}
  \begin{split}
    4\int_{\mathbb{R}^n} &\nabla_AuD^2\phi\overline{\nabla_Au}\,dx
    -\int_{\mathbb{R}^n}|u|^2\Delta^2\phi\,dx
     -2\int_{\mathbb{R}^n}|u|^2\phi'V_r\,dx
    \\
    &+4\int_{\mathbb{R}^n}u\phi'B_\tau\cdot\overline{\nabla_Au}\,dx  = \frac{d}{dt}
    \Im\int_{\mathbb{R}^n}\overline u\ \nabla_Au\cdot\nabla\phi\,dx
    = \ddot\Theta(t).
  \end{split}
\end{equation}
}

We give a proof of \eqref{eq:identity} in Section
\ref{sec:virial} for sufficiently smooth ($H^{3/2}$)
solutions, by a variant of
the classical Morawetz multiplier method. This approach has a long
history, starting with \cite{Morawetz68-a}
for the Klein-Gordon equation, \cite{Morawetz75-a},
\cite{Strauss75-a},
\cite{LionsPerthame92-a}; then the multiplier method was extended
to the Helmoltz and wave equations in \cite{PerthameVega99-a},
and for the
Schr\"odinger equation with an electric potential in
\cite{BarceloRuizVega97-a}, \cite{BarceloRuizVegaVilela-a}.
In the case of magnetic potentials, a 3D version of the
virial identity for Schr\"odinger first appeared in
\cite{Goncalves-Ribeiro91-b}, \cite{Goncalves-Ribeiro91-a}, while in
\cite{FanelliVega08-a} identity \eqref{eq:identity} is proved for
any dimension.

In order to apply the formal identity \eqref{eq:identity}
we shall need the following assumptions: the functions
$V(t,x)\in C^{1}$ and $A(t,x)=(A_{1},\dots,A_{n})\in C^{2}$ are
real valued, and for some constants $C,c>0$ and some $m\ge2$,
\begin{equation}\label{eq:ass1}
  c\xx^{m}\le V(t,x)\le C\xx^{m}, \qquad m\ge2;
\end{equation}
(but see Remark \ref{rem:gauge} below). Moreover we shall
assume that for some $m/2\le \lambda\le m-1$,
\begin{equation}\label{eq:ass2}
  (\partial_{r}V)^{+}\le C\xx^{m-1}\qquad
  |\nabla \cdot B_{\tau}|\le C\xx^{\lambda},\qquad
  |B_{\tau}|\le C\xx^{\lambda-m/2}
\end{equation}
where $(\partial_{r}V)^{+}$ is the positive part of the radial
derivative $\partial_{r}V$.

Recall that for superquadratic, time dependent potentials,
the existence of the propagator is still partially an
open question.
Hence we prefer to add an abstract, albeit very
natural, assumption concerning the well-posedness of the Cauchy
problem \eqref{eq:schro}:

\textbf{Assumption (H): well posedness.}
\textit{For each $t\in[-T,T]$, the operator
\begin{equation}\label{eq:opH}
  H(t)=-\left(\nabla-iA(t,x)\right)^2+V(t,x)
\end{equation}
is essentially selfadjoint on $C^{\infty}_{0}$,
with maximal domain $D(H(t))=D(H)$
independent of $t$; we shall use the notation
\begin{equation}\label{eq:HHs}
  \HH^{s}=D(H(t)^{s/2}),\qquad 0\le s\le2.
\end{equation}
Moreover, we assume that for each $f\in L^{2}$ problem
\eqref{eq:schro} has a solution $u\in C([-T,T],L^{2})$, which is
in $u\in C([-T,T],\HH^{1})$ for $f\in\HH^{1}$, and satisfies the
estimates
\begin{equation}\label{eq:est}
  \|u(t)\|_{L^{2}}\le C_{T}\|f\|_{L^{2}},\qquad
  \|u(t)\|_{\HH^{1}}\le C_{T}\|f\|_{\HH^{1}},\qquad
  t\in[-T,T].
\end{equation}
Finally, we assume that for $C^{\infty}_{0}$ data the solution
is at least in $C([-T,T],\HH^{3/2})$.
}

Notice that if $V,A$ do not depend on time, Assumption (H)
is trivially satisfied as soon as the operator $H$
is selfadjoint. As for the general case of superquadratic,
time dependent potentials, the optimal conditions for
well posedness are not clear. Some partial
results in this direction have been obtained by
Yajima in \cite{Yajima08-a},
where a propagator is constructed under condition slightly
more restrictive than \eqref{eq:ass1}, \eqref{eq:ass2}
(in particular, quadratic bounds for $\partial_{t}V,\partial_{t}A$
are required).

In the classical Morawetz estimates the tangential component
of $\nabla u$ satisfies better estimates than the
full gradient. A similar phenomenon occurs in presence of
a magnetic potential; we need to define here
the \emph{modified radial} and \emph{tangential derivatives} of $u$ as
\begin{equation}\label{eq:radtan}
  \nabla^{R}_{A}u=\frac{x}{|x|}\cdot \nabla_{A} u,\qquad
  \nabla^{T}_{A}u=\nabla_{A}u-\frac{x}{|x|}\nabla^{R}_{A}u
\end{equation}
with $\nabla_{A}=\nabla-iA(t,x)$, so that
\begin{equation}\label{eq:tan}
  |\nabla^{T}_{A}u|^{2}=
  \sum_{j<k}
  \left|\frac{x_{j}}{r}(\partial_{k}-iA_{k})u
  -\frac{x_{k}}{r}(\partial_{j}-iA_{j})u\right|^{2}.
\end{equation}
Notice that
\begin{equation*}
  |\nabla u|^{2}=|\nabla^{R}_{A}u|^{2}+|\nabla^{T}_{A}u|^{2}
\end{equation*}
and indeed $\nabla^{T}_{A}u$ reduces to the usual tangential
derivative when $A\equiv0$.

We are in position to state the main result of the paper:

\begin{theorem}\label{the:smoo}
  Let $n\ge3$, and assume that \eqref{eq:ass1}, \eqref{eq:ass2} and (H)
  hold for some $T>0$.
  Then for all data $f\in\HH^{1-1/m}$ the solution $u(t,x)$ of
  problem \eqref{eq:schro} satisfies for all $R>0$, with
  a constant $C$ independent of $R$,
  the following smoothing estimates:
  when $n\ge4$
  \begin{equation}\label{eq:smoo4D}
        \int_{-T}^{T}\int
        \left[
          \frac{R^{n-1}|\nabla_{A}u|^{2}}{(R \vee r)^{n}}
          +\frac{|\nabla^{T}_{A}u|^{2}}{r}
          +\frac{|u|^{2}}{r^{3}}
        \right]dxdt
         +\frac{1}{R^{2}}\int_{-T}^{T}\int_{|x|=R}|u|^{2}d \sigma dt
         \le
         C \|f\|^{2}_{\HH^{1-\frac1m}}
  \end{equation}
  while for $n=3$
  \begin{equation}\label{eq:smoo3D}
        \int_{-T}^{T}dt\int
        \left[
          \frac{R^{2}|\nabla_{A}u|^{2}}{(R \vee r)^{3}}
          +\frac{|\nabla^{T}_{A}u|^{2}}{r}
        \right]dx
         +\frac{1}{R^{2}}\int_{-T}^{T}dt\int_{|x|\le R}|u|^{2}d \sigma
         \le
         C \|f\|^{2}_{\HH^{1-\frac1m}}.
  \end{equation}
  If in addition we assume that $V$ is repulsive,
  i.e., $V_{r}\le0$, we can improve the above estimate bu replacing
  the $\HH^{1-1/m}$ norm at the right hand side with $\HH^{\lambda/m}$.
\end{theorem}

\begin{remark}\label{rem:gauge}
  In assumption \eqref{eq:ass1} we require a growth condition on $V$
  from below; this was one of the original assumptions of Yajima-Zhang
  \cite{YajimaZhang04-a} for $V=V(x)$, and was relaxed to
  \begin{equation}\label{eq:ass3}
    -C\xx^{m}\le V(t,x)\le C\xx^{m}
  \end{equation}
  (plus the corresponding ones for all derivatives
  $\partial_{x}^{\alpha} V$)
  in Robbiano-Zuily \cite{RobbianoZuily06-a}.
  We prefer to keep here this quite restrictive condition, since it
  makes it easier to deal with the spaces $\HH^{s}$ used in the
  statement of our result.
  Actually, we can reduce any potential satisfying \eqref{eq:ass3}
  to our situation by applying
  the time-dependent change of gauge
  \begin{equation}\label{eq:gauge}
    u(t,x)=e^{-i c_{0}t\xx^{m}}w(t,x)
  \end{equation}
  which transforms the equation into
  \begin{equation}\label{eq:newschro}
    iw_t(t,x)-(\nabla-i\widetilde{A}(t,x))^2w
    +\widetilde{V}(t,x)w(t,x)=0
  \end{equation}
  with
  \begin{equation}\label{eq:newpot}
    \widetilde{V}=V+c_{0}\xx^{m},\qquad
    \widetilde{A}=A+c_{0}\nabla\xx^{m}\cdot t.
  \end{equation}
  It is easy to check that the other assumptions remain true,
  with different constants; notice in particular that
  the field $B$ is unchanged.
\end{remark}


\section{Proof of the magnetic virial identity}\label{sec:virial}  

Let $u\in\mathcal H^{\frac32}$ be  a solution of
\eqref{eq:schro}. Recall that the quantity $\Theta_S(t) $
is defined as
\begin{equation*}
\Theta_S(t) =\int\phi|u(t,x)|^{2}dx
\end{equation*}
where the radial weight function $\phi$ will be
chosen in the following.
Writing equation \eqref{eq:schro} in the form
\begin{equation}\label{eq.schro2}
  u_t=-iHu,
\end{equation}
we obtain immediately
\begin{equation}
  \dot\Theta_S(t)  =-i\left\langle u,[H,\phi]u\right\rangle,\qquad
  \ddot\Theta_S(t)  =\left\langle u,[H,[H,\phi]]u\right\rangle,
  \label{eq.teta1S}
\end{equation}
where the brackets $[,]$ are the commutator and the brackets
$\langle,\rangle$ are the hermitian product in $L^2$. In order to
simplify the notations, we shall write
\begin{equation}\label{eq.T}
  T=-[H,\phi].
\end{equation}
By the Leibnitz formula
\begin{equation}\label{eq.leib}
  \nabla_A(fg)=g\nabla_Af+f\nabla g,
\end{equation}
which implies
\begin{equation}\label{eq.leibH}
  H(fg)=(Hf)g+2\nabla_Af\cdot\nabla g+f(\Delta g),
\end{equation}
we can write explicitly
\begin{equation}\label{eq.T2}
  T=2\nabla\phi\cdot\nabla_A+\Delta\phi.
\end{equation}
Observe that $T$ is anti-symmetric, namely
\begin{equation*}
  \langle f,Tg\rangle=-\langle Tf,g\rangle.
\end{equation*}
Hence we can rewrite \eqref{eq.teta1S} in the following form
\begin{equation}\label{eq.2teta2S}
  \ddot\Theta_S(t)=\left\langle u,[H,T]u\right\rangle,
\end{equation}
where $T$ is given by \eqref{eq.T2}.

In the following we shall use the shorthand notations,
for a function $f:\mathbb{R}^n\to\mathbb{C}$,
\begin{equation*}
  f_j =\frac{\partial f}{\partial x^j}, \qquad
  f_{\widetilde j}  =f_j-iA^jf, \qquad
  f_{\widetilde j^\star}  =f_j+iA^jf.
\end{equation*}
With these notations we have
\begin{equation*}
  (fg)_{\widetilde j}=f_{\widetilde j}g+fg_j
\end{equation*}
while the integrations by parts formula can be written
\begin{equation*}
  \int_{\mathbb{R}^n}f_{\widetilde j}(x)g(x)\,dx=-\int_{\mathbb{R}^n}f(x)g_{\widetilde
  j^\star}(x)\,dx.
\end{equation*}

We now compute explicitly the commutator $[H,T]$; by \eqref{eq.T2}
we have
\begin{equation}\label{eq.HT}
  [H,T]=-[\nabla_A^2,2\nabla\phi\cdot\nabla_A]
  -[\nabla_A^2,\Delta\phi]+[V,T]=:I+II+III.
\end{equation}
The term $III$ is easy:
\begin{equation}\label{eq.nuovaV}
  III=[V,T]=2[V,\nabla_A\cdot\nabla]=-2\nabla\phi\cdot\nabla V=-2\phi'V_r.
\end{equation}
As to $I$, we have
\begin{equation}
  \begin{split}
  -I = & 2\sum_{j,k=1}^n\left(\partial_{\widetilde j}\partial_{\widetilde j}\phi_k\partial_{\widetilde k}
  -\phi_k\partial_{\widetilde k}\partial_{\widetilde j}\partial_{\widetilde j}\right)
  \\
  = & \sum_{j,k=1}^n\left(2\phi_{kjj}\partial_{\widetilde k}+4\phi_{jk}
  \partial_{\widetilde j}\partial_{\widetilde k}+2\phi_k(\partial_{\widetilde j}\partial_{\widetilde j}
  \partial_{\widetilde k}-\partial_{\widetilde k}\partial_{\widetilde j}\partial_{\widetilde
  j})\right).
  \end{split}
  \label{eq.Iuno}
\end{equation}
Notice that
\begin{equation*}
  \begin{split}
  \partial_{\widetilde j}\partial_{\widetilde k}-\partial_{\widetilde k}\partial_{\widetilde j}= &
  i\left(A^j_k-A^k_j\right),
  \\
  \partial_{\widetilde j}\partial_{\widetilde j}\partial_{\widetilde k}
  -\partial_{\widetilde k}\partial_{\widetilde j}\partial_{\widetilde
  j} = &
  i\left(A^k_j-A^j_k\right)_j+2i\left(A^k_j-A^j_k\right)\partial_{\widetilde
  j};
  \end{split}
\end{equation*}
hence, by \eqref{eq.Iuno} we obtain
\begin{equation}\label{eq.Idue}
  -I=\sum_{j,k=1}^n\left(2\phi_{kjj}\partial_{\widetilde k}+4\phi_{jk}
  \partial_{\widetilde j}
  \partial_{\widetilde k}+2i\phi_j
  \left(A^j_k-A^k_j\right)_k+4i\phi_j
  \left(A^j_k-A^k_j\right)\partial_{\widetilde
  k}\right)
\end{equation}

The term $II$ can be written
\begin{equation}
  \begin{split}
  -II & =
  \sum_{j,k=1}^n\left(\partial_{\widetilde k}\partial_{\widetilde k}\phi_{jj}-
  \phi_{jj}\partial_{\widetilde k}\partial_{\widetilde k}\right)
  \\
  & =\sum_{j,k=1}^n\left(\phi_{jjkk}+2\phi_{jjk}\partial_{\widetilde k}\right).
  \end{split}
  \label{eq.II}
\end{equation}
By \eqref{eq.Idue} and \eqref{eq.II} we have
\begin{equation}
  \begin{split}
  \langle u,[\nabla_A^2,T]u\rangle = &
  \sum_{j,k=1}^n\int_{\mathbb{R}^n}\left(2u\phi_{kjj}\overline{u_{\widetilde k}}
  +4u\phi_{jk}\overline{\partial_{\widetilde j}\partial_{\widetilde k}u}
  +2u\phi_{kjj}\overline{u_{\widetilde k}}\right)\,dx
  \\
  & +\sum_{j,k=1}^n\int_{\mathbb{R}^n}\left(2i\phi_j\left(A^j_k-A^k_j\right)_k|u|^2
  +4iu\phi_j\left(A^j_k-A^k_j\right)\overline{u_{\widetilde
  k}}\right)\,dx
  \\
  & +\int_{\mathbb{R}^n}|u|^2\Delta^2\phi\,dx.
  \end{split}
  \label{eq.III}
\end{equation}
Using the identity
\begin{equation*}
  \overline{\partial_{\widetilde j}\partial_{\widetilde k}u}=
  \partial_{\widetilde j^\star}\partial_{\widetilde
  k^\star}\overline u
\end{equation*}
integrating by parts the first three terms of
\eqref{eq.III} we have
\begin{equation}
  \begin{split}
  \sum_{j,k=1}^n\int_{\mathbb{R}^n} & \left(2u\phi_{kjj}\overline{u_{\widetilde k}}
  +4u\phi_{jk}\overline{\partial_{\widetilde j}\partial_{\widetilde k}u}
  +2u\phi_{kjj}\overline{u_{\widetilde k}}\right)\,dx
  \\
  & =\sum_{j,k=1}^n\int_{\mathbb{R}^n}-4
    u_{\widetilde j}\phi_{jk}\overline{u_{\widetilde
  k}}\,dx=-4\int_{\mathbb{R}^n}\nabla_A uD^2\phi\overline{\nabla_A u}\,dx.
  \end{split}
  \label{eq.IV}
\end{equation}
For the 4th and 5th term in \eqref{eq.III} we notice that
\begin{equation*}
  \sum_{j,k=1}^n\phi_{jk}\left(A^j_k-A^k_j\right)=0,
\end{equation*}
and integrating by parts we obtain
\begin{equation}
  \begin{split}
  \sum_{j,k=1}^n\int_{\mathbb{R}^n} & \left(2i\phi_j\left(A^j_k-A^k_j\right)_k|u|^2
  +4iu\phi_j\left(A^j_k-A^k_j\right)\overline{u_{\widetilde
  k}}\right)\,dx
  \\
  & =4\Im\sum_{j,k=1}^n\int_{\mathbb{R}^n}u\phi_j
  \left(A^j_k-A^k_j\right)\overline{u_{\widetilde
  k}}\,dx
  \\
  & =4\Im\int_{\mathbb{R}^n}u\phi'B_\tau\cdot\overline{\nabla_{A}u}\,dx,
  \end{split}
  \label{eq.V}
\end{equation}
with $B_\tau$ as in Definition \ref{def.B}.

Collecting \eqref{eq.nuovaV}, \eqref{eq.III}, \eqref{eq.IV},
\eqref{eq.V} we conclude that
\begin{equation}
  \begin{split}
  \langle u,[H,T]u\rangle = &
  4\int_{\mathbb{R}^n}\nabla_A uD^2\phi\overline{\nabla_A u}
  -\int_{\mathbb{R}^n}|u|^2\Delta^2\phi
  \\
  & -2\int_{\mathbb{R}^n}\phi'V_r|u|^2
  +4\Im\int_{\mathbb{R}^n}u\phi'B_\tau\cdot\overline{\nabla_{A}u}.
  \end{split}
  \label{eq.HT2}
\end{equation}
Identities \eqref{eq.2teta2S} and \eqref{eq.HT2} imply
\eqref{eq:identity}.

\begin{remark}\label{rem.reg}
  Notice that, in order to justify all of the above
  computations, it is sufficient to require that the
  solution $u$ belongs to $\HH^{3/2}$
  (recall Assumption (H)); indeed, the
  highest order term is of the form
  \begin{equation*}
    \int\nabla_A^2u\nabla\phi\cdot\overline{\nabla_Au}.
  \end{equation*}
\end{remark}


\section{Choice of the multiplier}\label{sec:multiplier}  

The precise form of the multiplier $\phi$ will depend on the space
dimension. Writing $r=|x|$, we introduce the radial function
\begin{equation}\label{eq:phi}
  \phi_0(r)=\int_0^r\phi_0'(r)\,dx,
\end{equation}
where
\begin{equation}\label{eq:phi1}
  \phi_0'(r)=
  \begin{cases}
    M+\frac{1}{2n}r-\frac{1}{2n(n+2)}r^3,
    \qquad r\leq1
    \\
    M+\frac{1}{2n}-\frac{1}{2n(n+2)}\cdot\frac1{r^{n-1}},
    \qquad r>1,
  \end{cases}
\end{equation}
for some constant $M\geq1$. Hence we have also
\begin{equation}\label{eq:phi2}
  \phi_0''(r)=
  \begin{cases}
    \frac{1}{2n}-\frac{3}{2n(n+2)}r^2,
    \qquad r\leq1
    \\
    \frac{n-1}{2n(n+2)}\cdot\frac1{r^{n}},
    \qquad r>1.
  \end{cases}
\end{equation}
Observe that both $\phi_0'(r)$ and $\phi_0''(r)$ are positive and
continuous on $[0,+\infty)$. In order to compute
$\Delta^2\phi_0(|x|)$, we start by the laplacian, using the formula
\begin{equation*}
  \Delta\phi_0(r)=r^{1-n}\partial_r(r^{n-1}\phi_0'(r)),
\end{equation*}
which gives
\begin{equation}\label{eq:deltaphi}
  \Delta\phi_0(r)=
  \begin{cases}
    M(n-1)\cdot\frac1r+\frac12-\frac1{2n}r^2,
    \qquad r\leq1
    \\
    M(n-1)\cdot\frac1r+\frac{n-1}{2n}\cdot\frac1r,
    \qquad r>1;
  \end{cases}
\end{equation}
also $\Delta\phi_0(r)$ is continuous on
$[0,+\infty)$. Now we can compute the bi-laplacian using the formula
\begin{equation}\label{eq:bidelta}
  \Delta^2\phi_0(r)=r^{1-n}\partial_r(r^{n-1}\left(\Delta\phi_0\right)'(r)).
\end{equation}
Due to the presence of the function $1/r$ in \eqref{eq:deltaphi},
which is the fundamental solution of the laplacian in dimension
$n=3$, the cases $n\geq4$ and $n=3$ are slighlty different.

{\bf Case $n\geq4$.} By direct computation, from \eqref{eq:deltaphi}
we get
\begin{equation*}
  r^{n-1}\left(\Delta\phi_0\right)'(r)=
  \begin{cases}
    -M(n-1)r^{n-3}-\frac1nr^n,
    \qquad r\leq1
    \\
    -\left(M+\frac1{2n}\right)(n-1)r^{n-3},
    \qquad r>1.
  \end{cases}
\end{equation*}
Observe that $r^{n-1}\left(\Delta\phi_0\right)'(r)$ is discontinuous
at $r=1$, and the jump is given by
\begin{equation*}
 \left(\Delta\phi_0\right)'(1^+)-
 \left(\Delta\phi_0\right)'(1^-)=-\frac{n-3}{2n}.
\end{equation*}
As a consequence, \eqref{eq:bidelta} implies
\begin{equation}\label{eq:bideltaphi}  
\begin{split}
  \Delta^2\phi_0(r) = &
  -\left(1+\frac{M(n-1)(n-3)}{r^3}\right)\chi_{[0,1]}
  \\
  & -\left(M+\frac1{2n}\right)(n-1)(n-3)\cdot\frac{1}{r^3}\chi_{[1,+\infty)}
  -\frac{n-3}{2n}\delta_{r=1},
  \quad (n\geq4),
\end{split}
\end{equation}
where $\delta_{r=1}$ is the Dirac measure supported on the unit
sphere of $\mathbb{R}^n$. Notice that $\Delta^2\phi_0$ is negative.

{\bf Case $n=3$.} We rewrite \eqref{eq:deltaphi} as
\begin{equation}\label{eq:deltaphi3}
  \Delta\phi_0(r)=\varphi(r)+\psi(r),
\end{equation}
where
\begin{equation}\label{eq:fund}
  \varphi(r)=2M\cdot\frac1r,
\end{equation}
\begin{equation}\label{eq:psi}
  \psi(r)=
  \begin{cases}
    \frac12-\frac1{6}r^2,
    \qquad r\leq1
    \\
    \frac13\cdot\frac1{r^2},
    \qquad
    r\geq1.
  \end{cases}
\end{equation}
Clearly we have
\begin{equation*}
  \Delta\varphi(r)=-8\pi M\delta_{x=0},
\end{equation*}
\begin{equation*}
  \Delta\psi(r)=-\chi_{[0,1]},
\end{equation*}
where $\delta_{x=0}$ is the Dirac mass at the origin, and hence
\begin{equation}\label{eq:bidelta3}
  \Delta^2\phi_0(r)=-\chi_{[0,1]}-8\pi M\delta_{x=0},
  \qquad (n=3).
\end{equation}
Notice that also in this case the bilaplacian is negative.

We can now choose the multiplier $\phi$, which will be
defined as a suitable scaling of $\phi_0$: for any $R>0$
we set
\begin{equation}\label{eq:phifin}
  \phi_R(r)=R\phi_0\left(\frac{r}{R}\right).
\end{equation}
We have explicitly
\begin{equation}\label{eq:phi1scal}
  \phi_R'(r)=
  \begin{cases}
    M+\frac{1}{2n}\cdot\frac rR-\frac{1}{2n(n+2)}\cdot\frac{r^3}{R^3},
    \qquad r\leq R
    \\
    M+\frac{1}{2n}-\frac{1}{2n(n+2)}\cdot\frac{R^{n-1}}{r^{n-1}},
    \qquad r>R,
  \end{cases}
\end{equation}
\begin{equation}\label{eq:phi2scal}
  \phi_R''(r)=
  \begin{cases}
    \frac1R\left(\frac{1}{2n}-\frac{3}{2n(n+2)}\cdot\frac{r^2}{R^2}\right),
    \qquad r\leq R
    \\
    \frac1R\left(\frac{n-1}{2n(n+2)}\cdot\frac{R^n}{r^{n}}\right),
    \qquad r>R.
  \end{cases}
\end{equation}
Notice that $\phi'_R,\phi''_R$ are strictly positive
and more precisely
\begin{equation}\label{eq:posf1}
  \frac{\phi'}{r}\ge
  \begin{cases}
     \frac Mr+\frac{n-1}{2n(n+2)}\frac{1}{R}
              &\text{if $ r\le R $,}\\
     \frac Mr+\frac{n-1}{2n(n+2)}\frac{R^{n-1}}{r^{n}}
             &\text{if $ r\ge R $,}
  \end{cases}
\end{equation}
while
\begin{equation}\label{eq:posf2}
  {\phi''}\ge
  \begin{cases}
     \frac{n-1}{2n(n+2)}\frac{1}{R}
              &\text{if $ r\le R $,}\\
     \frac{n-1}{2n(n+2)}\frac{R^{n-1}}{r^{n}}
             &\text{if $ r\ge R $.}
  \end{cases}
\end{equation}
Moreover
\begin{equation}\label{eq:sup}
  \sup_{r\geq0}\phi'_R(r)=M+\frac1{2n},
  \qquad
  \sup_{r\geq0}\phi''_R(r)=\frac{1}{2nR}.
\end{equation}
The laplacian is given by
\begin{equation}\label{eq:deltaphifin}
  \Delta\phi_R(r)=
  \begin{cases}
    M(n-1)\cdot\frac1r+\frac1{2R}-\frac1{2n}\cdot\frac{r^2}{R^3},
    \qquad
    r\leq R
    \\
    M(n-1)\cdot\frac1r+\frac{n-1}{2n}\cdot\frac{1}{r},
    \qquad
    r>R
  \end{cases}
\end{equation}
whence in particular the estimate
\begin{equation}\label{eq:supD}
  |\Delta \phi_{R}|\le \frac{M(n-1)}{r}+\frac{1}{2(r \vee R)}.
\end{equation}
Also here the bilaplacian has a different form in the cases $n\geq4$ and
$n=3$. For $n\geq4$ we have
\begin{equation}\label{eq:bideltaphiscal}  
\begin{split}
  \Delta^2\phi_R(r) = &
  -\left(\frac1{R^3}+\frac{M(n-1)(n-3)}{r^3}\right)\chi_{[0,R]}
  \\
  & -\left(M+\frac1{2n}\right)(n-1)(n-3)\cdot\frac{1}{r^3}\chi_{[R,+\infty)}
  \\
  & -\frac{n-3}{2n}\cdot\frac{1}{R^2}\delta_{r=R},
  \qquad (n\geq4)
\end{split}
\end{equation}
while in dimension $n=3$ the bilaplacian is given by
\begin{equation}\label{eq:bideltascal3}
  \Delta^2\phi_R(r) = -\frac1{R^3}\chi_{[0,R]}-8\pi M\delta_{x=0},
  \qquad (n=3).
\end{equation}
Observe that in both cases the bilaplacian is negative.
In the following we shall drop the index $R$ and write simply
$\phi$ instead of $\phi_{R}$.

We can now plug these quantities into the identity
\eqref{eq:identity}. Let us consider the Hessian term on the L.H.S.
of \eqref{eq:identity}; using implicit summation over repeated indices, we
can write for a generic vector $v=(v_{1},\dots,v_{n})$
\begin{equation*}
  v \cdot D^{2}\phi \cdot v=
    \phi''(r)\left[\frac{x_{i}v_{i}}{r}\frac{x_{j}v_{j}}{r}
    \right]
    + \frac{\phi'(r)}{r}
    \left[v^{2}-\frac{x_{i}v_{i}}{r}\frac{x_{j}v_{j}}{r}
    \right]
\end{equation*}
with $v^{2}=v_{j}v_{j}$. Hence in particular
\begin{equation*}
  \nabla_{A}u\cdot D^{2}\phi \cdot  \overline{\nabla_{A}u}
  = \phi''\left|\frac{x}{|x|}\cdot \nabla_{A}u
  \right|^{2}
  + \frac{\phi'}{r}
  \left[|\nabla_{A}u|^{2}
  -\left|\frac{x}{|x|}\cdot \nabla_{A}u\right|^{2}
  \right].
\end{equation*}
Then the elementary identity
\begin{equation*}
  v^{2}w^{2}-(v \cdot w)^{2}=
  \sum_{i<j}(v_{i}w_{j}-v_{j}w_{i})^{2}
\end{equation*}
gives, recalling the notations \eqref{eq:radtan}, \eqref{eq:tan},
\begin{equation}\label{eq:hessA}
  \nabla_{A}u\cdot D^{2}\phi \cdot  \overline{\nabla_{A}u}=
  \phi''|\nabla^{R}_{A}u|^{2}+
  \frac{\phi'}{r}|\nabla^{T}_{A}u|^{2}.
\end{equation}
Now the identity \eqref{eq:identity} can be written
\begin{equation}
  \begin{split}
   4\int_{\mathbb{R}^n} & \phi''|\nabla^{r}_{A}u|^{2}\,dx+
    4\int_{\mathbb{R}^n}\frac{\phi'}{r}|\nabla^{T}_{A}u|^{2}\,dx
  -\int_{\mathbb{R}^n}|u|^2\Delta^2\phi\,dx
   -2\int_{\mathbb{R}^n}|u|^2 V_{r}\phi'\,dx
  \\
  &+4\int_{\mathbb{R}^n}
  (B_{\tau}\cdot\overline{\nabla_Au})\ \phi'u
  \,dx  = \frac{d}{dt}
  \Im\int_{\mathbb{R}^n}\overline u\ \nabla_Au\cdot\nabla\phi\,dx.
  \end{split}\label{eq:id3}
\end{equation}
Using \eqref{eq:posf1}, \eqref{eq:posf2} and the expressions for
$\Delta^{2}\phi$, we obtain the following estimates: for $n\ge4$
\begin{equation}\label{eq:4D}
  \begin{split}
    (n-1)(n-3)& M
     \int \frac{|u|^{2}}{r^{3}}dx+
    \frac{n-3}{2n}\frac{1}{R^{2}}\int_{|x|=R}|u|^{2}d \sigma+\\
    &   \qquad+
    \frac{2(n-1)}{n(n+2)}
    \int\frac{R^{n-1} }{(R \vee r)^{n}}|\nabla_{A}u|^{2}dx+
    2M\int \frac{|\nabla^{T}_{A}u|^{2}}{r}dx \le\\
     \le &
    2\int\phi'(V_{r})^{+}|u|^{2}dx+4\left|\int\phi'B_{\tau}\cdot
    \overline{\nabla_{A}u}\,udx\right|
    +\frac{d}{dt}
    \Im\int_{\mathbb{R}^n}\overline u\ \nabla_Au\cdot\nabla\phi\,dx
  \end{split}
\end{equation}
while for $n=3$ we have
\begin{equation}\label{eq:3D}
  \begin{split}
    8\pi M &|u(t,0)|^{2} +
      \frac{1}{R^{3}}\int_{|x|\le R}|u|^{2}dx+
    \frac{4}{15}
    \int\frac{R^{2} }{(R \vee r)^{3}}|\nabla_{A}u|^{2}dx+
    2M\int \frac{|\nabla^{T}_{A}u|^{2}}{r}dx \le\\
     \le &
    2\int\phi'(V_{r})^{+}|u|^{2}dx+4\left|\int\phi'B_{\tau}\cdot
    \overline{\nabla_{A}u}\,udx\right|
    +\frac{d}{dt}
    \Im\int_{\mathbb{R}^n}\overline u\ \nabla_Au\cdot\nabla\phi\,dx.
  \end{split}
\end{equation}


\section{Proof of Theorem \ref{the:smoo}}\label{sec:smoo}  

By the definition of $H(t)$ we have, for all $|t|\le T$,
\begin{equation}\label{eq:HH1}
  \|v\|_{\HH^{1}}^{2}=
  \|\nabla_{A} v\|^{2}_{L^{2}}+\int V|v|^{2}
  \gtrsim \|\xx^{m/2}v\|_{L^{2}}^{2}
\end{equation}
under our assumptions on $V(t,x)$.
Thus by interpolation we get

\begin{lemma}\label{lem:interp}
  For any $0\le \mu\le m/2$ and any $v\in\HH^{2\mu/m}$ we have
  \begin{equation}\label{eq:interp}
    \|\xx^{\mu} v\|_{L^{2}}\lesssim \|v\|_{\HH^{2\mu/m}}.
  \end{equation}
\end{lemma}

As a consequence, recalling the energy estimates \eqref{eq:est},
we have for any solution $u(t,x)$
\begin{equation}\label{eq:est2}
  \|\xx^{\mu} u\|_{L^{2}}\lesssim \|u\|_{\HH^{2\mu/m}}
  \le C_{T}\|u(0)\||_{\HH^{2\mu/m}}
\end{equation}
provided $0\le \lambda\le m/2$.

Also by interpolation we can prove the following bound which will
be used to estimate the left hand side in \eqref{eq:4D}, \eqref{eq:3D}:

\begin{lemma}\label{lem:h12}
  For any function $\phi\in C^{2}(\mathbb{R}^{n})$, such that
  \begin{equation}\label{eq:assphi}
    |\nabla \phi|+|x|\cdot
    |\Delta \phi|\le K,
  \end{equation}
  the following  inequality holds:
  \begin{equation}\label{eq:h12}
    \left|
    \int \overline f\ \nabla_Ag\cdot\nabla\phi\, dx
    \right|\le C(K)
    \|f\|_{\HH^{1/2}}\|g\|_{\HH^{1/2}}.
  \end{equation}
  Moreover, if $F(t,x)$ satisfies, for some $\frac{m}{2}\le \lambda\le m$
  \begin{equation}\label{eq:assB}
    \xx^{m/2}|F|+|\nabla F|
    \le K\xx^{\lambda}
  \end{equation}
  we have also
  \begin{equation}\label{eq:hl}
    \left|
    \int F(t,x)\cdot\overline f\ \nabla_Ag\cdot\nabla\phi\, dx
    \right|\le C(K)
    \|f\|_{\HH^{\lambda/m}}\|g\|_{\HH^{\lambda/m}}.
  \end{equation}
\end{lemma}

\begin{proof}
  Denote by $T(f,g)$ the bilinear operator
  \begin{equation*}
    T(f,g)=\int \overline f\ \nabla_Ag\cdot\nabla\phi\, dx.
  \end{equation*}
  By Cauchy-Schwartz we have immediately
  \begin{equation}\label{eq:bilin1}
    |T(f,g)|\le K \|f\|_{L^{2}}\|g\|_{\HH^{1}}.
  \end{equation}
  On the other hand, after an integration by parts, we have
  \begin{equation*}
    T(f,g)=-\int \overline{\nabla_{A}f}\cdot g \nabla\phi-
      \int \overline{f}g \Delta \phi
  \end{equation*}
  and again by Cauchy-Schwartz we get
  \begin{equation*}
    |T(f,g)|\le
      K\left\|\frac{f}{|x|}\right\|_{L^{2}}\|g\|_{L^{2}}
      +K\|f\|_{\HH^{1}}\|g\|_{L^{2}}.
  \end{equation*}
  Using the magnetic Hardy inequality (Theorem \ref{the:hardy}) this
  implies
  \begin{equation*}
     |T(f,g)|\le
        (K\cdot 2(n-2)^{-1}+K)
        \|f\|_{\HH^{1}}\|g\|_{L^{2}}.
  \end{equation*}
  By interpolation with \eqref{eq:bilin1} we obtain
  \eqref{eq:h12}.

  The proof of \eqref{eq:hl} is similar. Denoting again by $T(f,g)$
  the bilinear form at the left hand side of \eqref{eq:hl}, we have
  \begin{equation}\label{eq:Tfg}
    |T(f,g)|\le K^{2}\|\nabla_{A}g\|_{L^{2}}\|\xx^{\lambda-m/2}f\|_{L^{2}}
    \le K^{2}\|g\|_{\HH^{1}}\|f\|_{\HH^{2 \lambda/m-1 }}
  \end{equation}
  by \eqref{eq:interp}. Integrating by parts we have instead
  \begin{equation*}
    T(f,g)=-\int F \cdot\overline{\nabla_{A}f}\cdot g \nabla\phi -
      \int F \overline{f}g \Delta\phi-
      \int \nabla F \cdot \nabla\phi \overline{f}g=I+II+III.
  \end{equation*}
  The first term is equivalent to $T(g,f)$ and is estimated as above:
  \begin{equation*}
    |I|\le K^{2}\|f\|_{\HH^{1}}\|g\|_{\HH^{2 \lambda/m-1 }}.
  \end{equation*}
  Then, using the assumptions on $F,\phi$ we see that
  \begin{equation*}
    |II|\le
       K^{2}\|\xx^{\lambda-m/2}g\|_{L^{2}}\||x|^{-1}f\|_{L^{2}}\le
       C(K)\|g\|_{\HH^{2 \lambda/m-1 }}\|f\|_{\HH^{1}}
   \end{equation*}
   where we applied again the magnetic Hardy inequality
   \eqref{eq:hardy}. The third term gives
   \begin{equation*}
     |III|\le
       K^{2}\int\xx^{\lambda}|f||g|\le
       \|\xx^{\lambda-m/2}g\|_{L^{2}}\|\xx^{m/2}f\|_{L^{2}}\le
       C(K)\|g\|_{\HH^{2 \lambda/m-1 }}\|f\|_{\HH^{1}}.
   \end{equation*}
   In conclusion we have proved that
   \begin{equation*}
     |T(f,g)|\le
        C(K)\|g\|_{\HH^{2 \lambda/m-1 }}\|f\|_{\HH^{1}}
   \end{equation*}
   and by interpolation with \eqref{eq:Tfg} we obtain \eqref{eq:hl}.
\end{proof}

We can conclude the proof of the Theorem. In the case $n\ge4$,
it is clear that the left hand side of \eqref{eq:4D} is larger than
a multiple of
\begin{equation*}
  \int
  \left[
    \frac{R^{n-1}|\nabla_{A}u|^{2}}{(R \vee r)^{n}}
    +\frac{|\nabla^{T}_{A}u|^{2}}{r}
    +\frac{|u|^{2}}{r^{3}}
  \right]dx
   +\frac{1}{R^{2}}\int_{|x|=R}|u|^{2}d \sigma.
\end{equation*}
Hence, in order to obtain \eqref{eq:smoo4D}, it is sufficient (after
an integration on $[-T,T]$) to prove the following estimates:
\begin{equation}\label{eq:Vest}
  \int_{-T}^{T}
  \int\phi'(V_{r})^{+}|u|^{2}dx dt\le
  C_{T}\|f|_{\HH^{1-1/m}}^{2},
\end{equation}
\begin{equation}\label{eq:Best}
  \left|\int_{-T}^{T}
  \int \phi'B_{\tau}\cdot
  \overline{\nabla_{A}u}\,u dxdt\right|\le
  C_{T}\|f|_{\HH^{\lambda/m}}^{2},
\end{equation}
\begin{equation}\label{eq:Imest}
  \left.
  \Im\int_{\mathbb{R}^n}\overline u\ \nabla_Au\cdot\nabla\phi\,dxdt
  \right|_{-T}^{T}\le
  C_{T}\|f|_{\HH^{1/2}}^{2}.
\end{equation}
In order to prove
the first estimate \eqref{eq:Vest}, we can write using
\eqref{eq:sup}, assumption \eqref{eq:ass2} on $V$ and the
inequality \eqref{eq:interp},
\begin{equation*}
  \int_{-T}^{T}
  \int\phi'(V_{r})^{+}|u|^{2}dx dt
    \le C \int_{-T}^{T}\|u(t)\|_{\HH^{1-1/m}}^{2}dt
    \le C_{T}\|f\|_{\HH^{1-1/m}}^{2}
\end{equation*}
where in the final step we applied the energy estimate
\eqref{eq:est2}. To prove the second estimate \eqref{eq:Best}, it is
sufficient to use \eqref{eq:hl} of Lemma \ref{lem:h12} with the
choice $F=B_{\tau}$, recalling assumptions \eqref{eq:ass2} on
$B_{\tau}$, the bounds \eqref{eq:sup}, \eqref{eq:supD} on $\phi$ and
using again the energy estimate \eqref{eq:est2}. Finally, the third
estimate \eqref{eq:Imest} is exactly \eqref{eq:h12} of Lemma
\ref{lem:h12}. Since $\lambda\le m-1$ and $m\ge2$, this concludes
the proof in the case $n\ge4$.

The proof in the case $n=3$ is completely analogous.



\appendix  
\section{Some technical lemmas}\label{sec:appendix}

\begin{theorem}[Magnetic Hardy Inequality]\label{the:hardy}
  Assume $A(x)=(A_{1},\dots,A_{n})$ is in $L^{2}_{\mathrm{loc}}$,
  with values in $\mathbb{R}^{n}$, $n\ge3$. Then for all
  $u$ in the domain of $\nabla_{A}^{2}=(\nabla-iA)^{2}$ the
  following inequality holds:
  \begin{equation}\label{eq:hardy}
    \int \frac{|u|^{2}}{|x|^{2}}dx\le
    \left(\frac{2}{n-2}\right)^{2}
    \int|\nabla_{A}u|^{2}dx.
  \end{equation}
\end{theorem}

\begin{proof}
  The proof is similar to the standard one for $A=0$. Indeed,
  for any $\alpha\in \mathbb{R}$ we have
  \begin{equation*}
    0\le\int\left|\nabla_{A}u+ \frac{\alpha x}{|x|^{2}}u\right|^{2}\equiv
    \int|\nabla_{A}u|^{2}+\alpha^{2}\int \frac{|u|^{2}}{|x|^{2}}+
    2 \alpha \mathbb{R}e\int \nabla_{A}u \cdot \frac{x}{|x|^{2}}
    \overline{u}.
  \end{equation*}
  We notice that
  \begin{equation*}
    2\alpha \mathbb{R}e\int \nabla_{A}u \cdot \frac{x}{|x|^{2}}
      \overline{u}=
    2\alpha \mathbb{R}e\int \nabla u \cdot \frac{x}{|x|^{2}}
      \overline{u}=
    \alpha\int \nabla|u|^{2}\cdot\frac{x}{|x|^{2}}
  \end{equation*}
  and integrating by parts we get
  \begin{equation*}
    0\le
    \int|\nabla_{A}u|^{2}+\alpha(\alpha-n+2)\int \frac{|u|^{2}}{|x|^{2}}.
  \end{equation*}
  Choosing $\alpha=(n-2)/2$ we conclude the proof.
\end{proof}








\end{document}